\documentclass[11pt]{amsart}
\usepackage{latexsym,amssymb,amsmath,youngtab}
\usepackage{youngtab}
\usepackage{epsfig}
\usepackage{amsmath,amsthm,amssymb,amscd,MnSymbol}
\usepackage[mathscr]{eucal}
\usepackage{verbatim, hyperref}
\usepackage{color}
\newcommand*{\textlabel}[2]{%
  \edef\@currentlabel{#1}
  \phantomsection
  #1\label{#2}
}

\newtheoremstyle{custom}
  {3pt}
  {3pt}
  {\slshape}
  {}
  {\bfseries}
  {.}
  { }
   {}
\theoremstyle{custom}
\newtheorem{theorem}{Theorem}[section]
\newtheorem{proposition}[theorem]{Proposition}

\newtheorem{proposition/definition}[theorem]{Proposition/Definition}
\newtheorem{lemma}[theorem]{Lemma}
\newtheorem{corollary}[theorem]{Corollary}

\newtheorem{prop}[theorem]{Proposition}

\theoremstyle{definition}
\newtheorem{definition}[theorem]{Definition}

\newtheorem{example}[theorem]{Example}

\theoremstyle{remark}
\newtheorem{remark}[theorem]{Remark}

\newcommand{\stack}[2]{\ensuremath{\genfrac{}{}{0pt}{}{#1}{#2}}} 

\newtheoremstyle{exercise}
  {3pt}
  {6pt}
  {}
  {}
  {\bfseries}
  {:}
  { }
   {}
\theoremstyle{exercise}
\newtheorem{exercise}[theorem]{Exercise}
\newtheoremstyle{exercises}
  {3pt}
  {6pt}
  {}
  {}
  {\bfseries}
  {:}
  {\newline}
   {}

\theoremstyle{exercise}
\newtheorem{exercises}[theorem]{Exercises}

\input epsf
\def\boxit#1{\vbox{\hrule height1pt\hbox{\vrule width1pt\kern3pt
  \vbox{\kern3pt#1\kern3pt}\kern3pt\vrule width1pt}\hrule height1pt}}

\def\trank{\text{rank}}

\def\BC{\mathbb C}\def\BF{\mathbb F}

\def\BP{\mathbb P}
\def\pp#1{\mathbb P^{#1}}
\def\fa{\mathfrak a}
\def\fb{\mathfrak b}

\def\pp#1{{\mathbb P}^{#1}}
\def\tdim{{\rm dim}}

\def\hd{,...,}
\def\ww{\wedge}
\def\upperp{{}^\perp}

\def\inv{{}^{-1}}

\def\cO{{\mathcal O}}
\def\CC{\mathbb C}

\def\11{\mathbf 1}

\def\fsl{{\mathfrak {sl}}}

\def\fm{{\mathfrak m}}

\def\l{\lambda}
\def\a{\alpha}

\def\o{\omega}

\def\b{\beta}
\def\g{\gamma}
\def\s{\sigma}

\def\ot{{\mathord{ \otimes } }}
\def\op{{\mathord{\,\oplus }\,}}
\def\otc{{\mathord{\otimes\cdots\otimes}\;}}

\def\ra{{\mathord{\;\rightarrow\;}}}

\def\dim{{\rm dim}\;}
\def\La#1{\Lambda^{#1}}

\def\frak{\mathfrak}
\def\fsl{\frak s\frak l}

\def\op{\oplus}
\def\BF{\Bbb F}\def\BZ{\Bbb Z}

\def\ep{\epsilon}
\def\op{\oplus}

\def\s{\sigma}
\def\t{\tau}

\def\a{\alpha}
\def\b{\beta}

\def\g{\gamma}
\def\l{\lambda}

\def\FS{\mathfrak  S}

\def\ol{\overline}

\def\BP{\mathbb  P}
\def\BC{\mathbb  C}

\def\pp#1{\mathbb  P^{#1}}

\def\ep{\epsilon}

\def\opc{\op\cdots\op}

\def\hd{, \hdots ,}

\def\inv{{}^{-1}}

\def\La#1{\Lambda^{#1}}

\def\pp#1{\mathbb  P^{#1}}

\def\ur{\underline {\bold R}}

\def\ra{\rightarrow}

\def\tdet{\operatorname{det}}

\def\tdim{\operatorname{dim}}

\def\tlim{\lim}

\def\tmin{\operatorname{min}}

\def\trank{\operatorname{rank}}

\def\upperp{{}^{\perp}}

\def\ww{\wedge}

\def\ctimes{\times \cdots\times}

\def\be{\begin{equation}}
\def\ene{\end{equation}}
\def\aaa{{\bold {a}}}

\newcommand{\Id}{\operatorname{Id}}

\def\tzeros{{\rm Zeros}}

\newcommand{\Spec}{\operatorname{Spec}}

\def\Mn{M_{\langle \nnn \rangle}}\def\Mthree{M_{\langle 3\rangle}}

\def\cK{{\mathcal K}}

\def\trank{{\mathrm {rank}}}
\def\len{{\mathrm{length}}}

\def\aaa{{\bold a}}

\def\VV{\mathbf{V}}

\newcommand{\Mat}{\operatorname{Mat}}

\def\nnn{\bold n}

\def\Mnred{M_{\langle \nnn \rangle}^{red}}
\def\Monred{M_{\langle 1,1, \nnn \rangle}^{red}}

\begin{document}

\title[Geometry of border rank algorithms]{On the geometry of border rank algorithms for matrix multiplication and other tensors with symmetry}
\author{J.M. Landsberg}
\address{
Department of Mathematics\\
Texas A\&M University\\
Mailstop 3368\\
College Station, TX 77843-3368, USA}
\email{jml@math.tamu.edu}
\author{Mateusz Micha{\l}ek}
\address{
Freie Universit\"at\\
 Arnimallee 3\\
 14195 Berlin, Germany\newline
Polish Academy of Sciences\\
         ul. \'Sniadeckich 8\\
         00-956 Warsaw\\
         Poland}
\email{wajcha2@poczta.onet.pl}

 \begin{abstract} We establish basic information about border rank algorithms
 for the matrix multiplication tensor and other tensors with symmetry. We prove that 
 border rank  algorithms for tensors with 
 symmetry (such as matrix multiplication and the determinant polynomial)   come in families that
 include representatives with normal forms. 
 These normal forms will be useful both to develop new efficient
 algorithms and to prove lower complexity bounds.
 We   derive a border rank version of the substitution
 method used in proving lower bounds for tensor rank. We use this border-substitution method and a normal form to improve the
 lower bound on the border rank of matrix multiplication by one, to $2\nnn^2-\nnn+1$. We also point out
 difficulties that will be formidable obstacles to future progress
 on lower complexity bounds for tensors because of the \lq\lq wild\rq\rq\ structure
 of the Hilbert scheme of points.
 \end{abstract}
\thanks{Landsberg    supported by   NSF grant  DMS-1405348. Michalek was supported by Iuventus Plus grant 0301/IP3/2015/73 of the Polish Ministry of Science.}

\keywords{matrix multiplication complexity, border rank, tensor, commuting matrices, Strassen's equations,   MSC 15.80}
\maketitle

\section{Introduction}

Ever since Strassen discovered in 1969 \cite{Strassen493} that
the standard algorithm for multiplying matrices is not optimal, it has
been a central question to determine upper and lower bounds for
the complexity of the matrix multiplication tensor
$\Mn\in \BC^{\nnn^2}\ot \BC^{\nnn^2}\ot \BC^{\nnn^2}$. In the language
of algebraic geometry, this amounts to determining the smallest
value $r$ such that the matrix multiplication tensor
lies on the $r$-th secant variety of the Segre variety
$Seg(\pp{\nnn^2-1}\times \pp{\nnn^2-1}\times \pp{\nnn^2-1})$-- see below
for definitions. This value of $r$ is called the {\it border rank} of $\Mn$
and is denoted $\ur(\Mn)$. 

The main contribution of this article is the observation that one can
simplify this study by restricting one's search to border rank algorithms
of a very special form. This special class of algorithms is of interest in its
own right and we develop basic language to study them.

From the perspective of algebraic geometry, our restriction
 amounts to reducing the
study of the Hilbert scheme of points to the punctual Hilbert scheme
(those schemes supported at a single point).
We expect it to be useful in other situations.

While motivated by the complexity of 
matrix multiplication, our work fits into both the larger study of 
the structure of secant varieties of homogeneous
varieties (e.g., \cite{zak,MR3239293}) and the study
of the geometry of tensors (e.g., \cite{MR2865915}).

\subsection*{Overview} In \S\ref{srdefs} we define  secant varieties and a
variety of border rank algorithms. In \S\ref{norformsect} we prove our main
normal form lemma and show that it applies to the problems
of studying the Waring border rank of the determinant and the
tensor border rank of the matrix multiplication operator.
To better study the normal forms, in \S\ref{areolesect}  we define subvarieties
of secant varieties
associated to certain constructions that have already
appeared in the literature \cite{2015arXiv151200609M,2015arXiv151105707B}.
In \S\ref{braftsect} we prove a border rank version of the substitution
method as  used in \cite{alexeev+forbes+tsimerman:2011:tensor-rank} and apply it to show
$\ur(\Mn)\geq 2\nnn^2-\nnn+1$, an improvement by one over
the previous lower bound of \cite{MR3376667}. 

\subsection*{Notation}
Throughout this paper,
$\VV, A,B,C,U,V,W$ denote complex vector
spaces and  $X\subset \BP \VV$ denotes a projective variety. 
If $v\in \VV$, we let $[v]\in \BP \VV$ denote the corresponding point
in projective space.
For a variety $X$, $X^{(r)}=X^{\times r}/\FS_r$ denotes the 
$r$-tuples of   points of $X$, where $\FS_r$ is the group of permutations on $r$ elements. The vector space of linear maps
$U\ra V$ is denoted $U^*\ot V$,
\section*{Acknowledgements}
We would like to thank Jaroslaw Buczy{\'n}ski and Joachim Jelisiejew for many interesting discussions on secant varieties and local schemes. We thank the Simons Institute for the Theory of Computing, UC
Berkeley, for providing a wonderful environment during the
program
{\it Algorithms and Complexity
in Algebraic Geometry}
during which work on this article began.
Michalek would like to thank PRIME DAAD program. 
 
\section{Secant varieties}\label{srdefs}

Let $X\subset \BP\VV$ be a variety,  let 
$$\s_r^0(X)= \bigcup_{x_1\hd x_r\in X}\langle x_1\hd x_r\rangle  \subset \BP\VV
$$ 
denote the points of $\BP\VV$ on secant $\pp{r-1}$'s of $X$, 
and let $\s_r(X):=\ol{\s_r^0(X)}$ denote its Zariski closure, 
the {\it $r$-th secant variety}  of $X$, where $\langle x_1\hd x_r\rangle$ denotes the projective linear space spanned by
the points $x_1\hd x_r$ (usually it is a $\pp{r-1}$).

In this paper we are primarily concerned with the case
$\BP\VV=\BP(A\ot B\ot C)$ and $X=Seg(\BP A\times \BP B\times \BP C)$ is
the {\it Segre variety} of rank one tensors. The above-mentioned
question about the matrix-multiplication tensor is the case
$A=U^*\ot V$, $B=V^*\ot W$, and $C=W^*\ot U$, and
$M_{\langle U,V,W\rangle}\in A\ot B\ot C$ is the matrix multiplication tensor. 
When $U,V,W=\BC^n$, we denote $M_{\langle U,V,W\rangle}$ by $\Mn$.
The question is: What is the smallest $r$ such that
$[M_{\langle U,V,W\rangle}]\in \s_r(Seg(\BP A\times \BP B\times \BP C))$?
Bini  \cite{MR605920} showed that this $r$, called the {\it border rank}
of $M_{\langle U,V,W\rangle}$ indeed governs its complexity. 
The border rank of a tensor $T$ is denoted $\ur(T)$. 
The smallest $r$ such that a tensor $[T]\in \BP (A\ot B\ot C)$ is
in $\s_r^0(Seg(\BP A\times \BP B\times \BP C))$ is called the {\it rank}
of $T$ and is denoted $\bold R(T)$.

\begin{remark} 
It is expected that the rank of $\Mn$ is greater than its
border rank when $\nnn>2$, and more generally we expect that for \lq\lq most\rq\rq\ tensors $T$ with a large
symmetry group, $\ur(T)<\bold R(T)$. In the case of matrix multiplication we have the following
evidence: $\ur(\Mn)\geq 2\nnn^2-O(\nnn)$ and $\bold R(\Mn)\geq 3\nnn^2-o(\nnn^2)$
\cite{MR3376667,MR3162411}. Moreover, $19\leq \bold R(\Mthree)\leq 23$ while $16\leq \ur(\Mthree)\leq 20$
with inequalities proved respectively in \cite{Bl2,Laderman}, this article,  and \cite{MR3146566}. 
\end{remark}

\begin{definition}
 Let $X\subset \BP \VV$ be a projective variety.
By an {\it $X$-border rank $r$ algorithm} for $z\in \BP \VV$, we   mean a curve $E_t$ in the
Grassmannian $G(r,V)$ such that $z\in \BP E_0$ and for $t>0$, $E_t$ is
spanned by $r$ points of $X$. (This includes the possibility of
$E_t$ being stationary.) In particular,   $z$ admits an $X$-border rank $r$ algorithm if and only if 
$z\in \s_r(X)$.  
When $X=Seg(\BP A\times \BP B\times \BP C)$, we just refer to
{\it border rank algorithms}. We will say such an $E_0$ {\it realizes} $z$ as point
of $\s_r(X)$. 
\end{definition} 

 \begin{remark}
 Instead of taking a curve in the Grassmannian we may and sometimes will  just take a convergent sequence $E_{t_n}$. 
 \end{remark}

Define 
the {\it incidence variety}
$$
S_r^0(X) :=\{ ([v],([x_1]\hd [x_r]))\mid v\in \langle x_1\hd x_r\rangle\}\subset \BP \VV\times X^{(r)} ,
$$
a \lq\lq Nash\rq\rq - type blow up of it
$$
\tilde S_r^0(X) :=\{ ([v],([x_1]\hd [x_r]),\langle x_1\hd x_r\rangle )\mid v\in \langle x_1\hd x_r\rangle, \tdim \langle x_1\hd x_r\rangle=r\}
\subset \BP \VV\times X^{(r)}\times G(r,\VV), 
$$
and the {\it abstract secant variety}
$$
S_r(X):=\ol{\tilde S_r^0(X)}.
$$
We have maps
\begin{align*}
    S_r&(X)   \\
 ^\rho\swarrow \ \  & \  \searrow ^\pi\\
   G(r,\VV) \ \ \ \  &  \  \ \ \s_r(X)
 \end{align*}
where the   map $\pi$ is surjective.

 When discussing rank algorithms, a point $p\in \s_r^0(X)$ is called
{\it identifiable} if there is a unique collection of $r$ points of $X$ such that
$p$ is in their span. When discussing border rank realizations, the $r$-plane
$E_0$ is the more important object, which motivates the following definition:

\begin{definition} We say $[v]\in \s_r(X)$ is {\it Grassmann-border-identifiable} if $\rho\pi\inv([v])$ is a point.
\end{definition}

We will be mostly interested in the case  
 when $X=G/P\subset \BP \VV$ is a homogeneous variety and  $[v]$ has a nontrivial symmetry group $G_v\subset G=G_X$.
In this case $[v]$ is almost never Grassmann-border-identifiable. Indeed, if $z\in \pi\inv([v])$, then the
orbit closure $\overline{G_v\cdot z}$ is also in $\pi\inv([v])$. Hence, to be Grassmann-border-identifiable, $G_v$ would have to act trivially on $\rho(z)$.

\section{The normal form lemma}\label{norformsect}

By \cite[Lemma 2.1]{MR3239293} in any border rank algorithm with $X=G/P$ homogeneous, we may assume there is one stationary point
$x\in X$ with $x\in \BP E_t$ for all $t$.

The following Lemma is central: 

\begin{lemma}[Normal form lemma]\label{normalformlem}  Let $X=G/P\subset \BP \VV$ and let $v\in \VV$ be such that $G_v$ has a single closed orbit
$\cO_{min}$ in $X$. 
Then the $G_v$-orbit closure of any border rank $r$ algorithm of $v$ contains a border rank $r$ algorithm $E=\tlim_{t\ra 0}\langle
x_1(t)\hd x_r(t)\rangle$ where there is
a stationary point $x_1(t)\equiv x_1$  lying in $\cO_{min}$.

If moreover  every orbit of $G_{v,x_1}$ contains $x_1$ in its closure, we may
further assume that   all other $x_j(t)$ limit to $x_1$.
\end{lemma}

\begin{proof}  
The proof of the first statement follows from the same methods as the proof of the second, hence we focus on the latter. 
We prove we can have all points limiting to the same point $x_1(0)$. By \cite[Lemma 2.1]{MR3239293}  this is enough to conclude. 

We work by induction. Say we have shown that $x_1(t)\hd x_q(t)$ all limit to the same point $x_1\in\cO_{min}$.
We will show that our curve can be modified so that the  same holds for $x_1(t)\hd x_{q+1}(t)$.
Take a curve $g_{\ep}\in G_{v,x_1}$ such that $\tlim_{\ep\ra 0}g_\ep x_{q+1}(0)=x_1$. For each fixed $\ep$, acting 
on the $x_j(t)$ by $g_\ep$,  we obtain a border rank algorithm for which $g_\ep x_i(t)\rightarrow x_1(0)$ for
$i\leq q$ and $g_\ep x_{q+1}(t)\rightarrow g_{\ep}x_{q+1}(0)$. Fix a sequence $\ep_n\rightarrow 0$. 
Claim: we may choose a sequence $t_n\ra 0$ such that
\begin{itemize}
\item $\lim_{n\rightarrow\infty}g_{\ep_n}x_{q+1}(t_n)= x_1(0)$, and 
\item $\lim_{n\rightarrow\infty}<g_{\ep_n}x_1(t_n),\dots,g_{\ep_n}x_r(t_n)>$ contains $v$.
\end{itemize}
The first   point holds as $\tlim_{\ep\ra 0}g_\ep x_{q+1}(0)=x_1$. The second  follows as for each fixed $\ep_n$, taking $t_n$ sufficiently small we may assure that a ball of radius $1/n$ centered at $v$ intersects $<g_{\ep_n}x_1(t_n),\dots,g_{\ep_n}x_r(t_n)>$.
Considering the sequence  $\tilde x_i(t_n):=g_{\ep_n}x_i(t_n)$ we obtain the desired border rank algorithm. 
\end{proof}

Our main interest consists of  the following two examples:

\subsection{The determinant polynomial}Let  $v_n: \BP W\ra \BP (S^nW)$ denote  the Veronese re-embedding of $\BP W$.
When  $W=E\ot F=\BC^n\ot \BC^n$, the space $\VV=S^nW$ is the  home of the determinant polynomial. Write 
$X=v_n(\BP W)\subset \BP S^nW$   and $v=\tdet_n$ for  the determinant. 
Here $G_X=GL_{n^2}$ and $G_v\simeq (SL(E)\times SL(F))\ltimes \BZ_2$. The group 
$G_v$ has a unique closed orbit $\cO_{min}=v_n(Seg(\BP E\times \BP F))$ in $X$.  
Moreover,
for any $z\in v_n(Seg(\BP E\times \BP F))$, $G_{\tdet_n,z}$, the group preserving
both $\tdet_n$ and $z$,  is isomorphic to $P_E\times P_F$, where $P_E,P_F$ are the parabolic 
subgroups of matrices with zero in the first column except the $(1,1)$-slot, and  
  $z$ is  in the $G_{\tdet_n,z}$-orbit closure of any $q\in v_n(\BP W)$.

\subsection{The matrix multiplication tensor} 
Set $A=U^*\ot V$, $B=V^*\ot W$, $C=W^*\ot U$. The space $\VV=A\ot B\ot C$ is the  home of the matrix multiplication tensor, 
$X=Seg(\BP A\times \BP B\times \BP C)=Seg(\BP (U^*\ot V)\times \BP (V^*\ot W)\times \BP (W^*\ot U))\subset \BP (A\ot B\ot C)$, 
and $v=M_{\langle U,V,W\rangle}=\Id_U\ot \Id_V\ot \Id_W\in (U^*\ot V)\ot (V^*\ot W)\ot (W^*\ot U)$ is the matrix multiplication
tensor. ($\Id_U\in U^*\ot U$ denotes the identity map and in the
expression for $M_{\langle U,V,W\rangle}$ we re-order factors.) Here $G_X=GL(A)\times GL(B)\times GL(C)$ 
and $G_{M_{\langle U,V,W\rangle}}=GL(U)\times GL(V)\times GL(W)$, and both are slightly larger by
a finite group if some of the  dimensions coincide.

\begin{proposition}\label{ckprop}
Let 
$$\cK:=\{
[\mu\ot v\ot \nu\ot w\ot \o\ot u]\in 
 Seg(\BP U^*\times \BP V\times \BP V^*\times \BP W\times \BP W^*\times \BP U)
 \mid \mu(u)=\o(w)=\nu(v)=0
 \}
 $$
Then $\cK$ is the unique closed $G_{M_{\langle U,V,W\rangle}}$-orbit in $Seg(\BP A\times \BP B\times \BP C)$.

Moreover, if $k\in \cK$, then $G_{M_{\langle U,V,W\rangle},k}$, the group preserving both $M_{\langle U,V,W\rangle}$
and $k$, is such that for every $p\in Seg(\BP A\times \BP B\times \BP C)$, $k\in \ol{G_{M_{\langle U,V,W\rangle},k}\cdot p}$.
\end{proposition}

Note that $Seg(\BP U\times \BP U^*)_0:=\{ [u\ot \a]\mid \a(u)=0\}\subset \BP \fsl(U)$ is the closed orbit in the adjoint
representation and $\cK$ is isomorphic to $Seg( Seg(\BP U\times \BP U^*)_0\times  Seg(\BP V\times \BP V^*)_0\times  Seg(\BP W\times \BP W^*)_0 )$.

\begin{proof}
It is enough to prove the last statement. We will prove that $k$ is the unique closed orbit under $G_{M_{\langle U,V,W\rangle},k}$.
This is enough to conclude as the closure of any orbit must contain a closed orbit. Notice that fixing $k=[( \mu\ot  v)\ot ( \nu\ot  w)\ot ( \o \ot  u)]$ is equivalent to fixing a partial flag in each $U,V$ and $W$ consisting of a line and a hyperplane containing it.

Let $[a\ot b\ot c]\in   Seg(\BP A\times \BP B\times \BP C)$. If $[a]\not\in Seg(\BP U^*\times \BP V)$ then the orbit is not closed, even under the torus action on $V$ that is compatible with the flag. 
So without loss of generality, we may assume $[a\ot b\ot c]\in Seg(\BP U^*\times \BP V\times \BP V^*\times \BP W\times \BP W^*\times \BP U)$.
Write $a\ot b\ot c=(\mu'\ot v')\ot (\nu'\ot w')\ot (\o' \ot u')$.
If, for example $ v'\neq v$, we may act with an element of $GL(V)$ that preserves the partial flag and sends $v'$ to $ v+\epsilon v'$. Hence 
$v$ is  in the closure of the orbit of $v'$. As $G_{M_{\langle U,V,W\rangle},k}$ preserves $ v$ we may continue, reaching   $k$ in the closure.
\end{proof}

\begin{remark} Proposition \ref{ckprop} combined with the normal form Lemma
allows the argument of \cite{MR2188132} to be simplified tremendously, as it vastly reduces the
number of cases. In particular, it eliminates the need for the erratum.
\end{remark}

\section{Local versions of secant varieties}\label{areolesect}
In this section we introduce several higher order generalizations
of the tangent star at a point of a variety and the tangential variety.
The generalizations are  subvarieties of the $r$-th  secant variety  of a projective variety.
We restrict our discussion to projective varieties $X\subset \BP \VV,$ 
however the discussion can be extended to arbitrary embedded schemes.
Our initial motivation was to provide language to discuss the normal
form of the main lemma, but   the discussion is  useful
in a wider context;   special cases have already been
used in  \cite{2015arXiv151105707B, 2015arXiv151200609M}.

We exhibit local properties of the $r$-th secant variety 
using the language of  smoothable schemes of length $r$ supported at one point (i.e.~local). Their moduli space is in the principal component of the Hilbert scheme of subschemes of length $r$ of $X$. This component is an algebraic variety,   a compactification of $r$-tuples of distinct points of $X$, i.e.~$(X^{( r)}\setminus D)$, where $D$ is the big diagonal. It parametrizes smoothable schemes, i.e.~schemes that arise as degenerations of of $r$ distinct points with reduced structure.  More formally, an ideal $I$ defines a smoothable scheme if there exists a flat family $I_t$ with the fiber $I$ for $t= 0$, where
for $t\neq 0$, $I_t$ is the ideal of $r$ distinct points.

\subsection{Areoles and buds}
Recall that a scheme $S$ supported   at $0\in \BC^n$ corresponds to an ideal
$I\subset \BC[x_1\hd x_n]$ whose only zero is $(0)$. Define the span of $S$ to be
$\langle S\rangle := \tzeros (I_1)$ where $I_1\subset I$ is the homogeneous degree
one component. This definition depends on the embedding of $S$.

We start by recalling the definition of the areole from \cite[Section 5.1]{2015arXiv151105707B}.
\begin{definition}[Areole]
Let $p\in X$.
   The \textbf{$r$-th open areole} at $p$ is
   \begin{align*}
     \mathfrak{a}^\circ_r(X,p)&:= { \bigcup \{ \langle R \rangle \mid   
             R\text{ is smoothable in }X \text{, supported at } p \text{ and } \len(R) \leq r \} },
   \end{align*}
     the \textbf{$r$-th areole} at $p$ is the closure:
   \begin{align*}
     \mathfrak{a}_r(X,p)&:= \overline{ \mathfrak{a}^\circ_r(X,p)},
   \end{align*}
 and the \textbf{$k$-th areole variety} of $X$ is
$$\fa_r(X):=\bigcup_{p\in X}\mathfrak{a}_r(X,p).
$$
\end{definition}
The areole can be regarded as a   generalization of a tangent space. 
Indeed, consider $r=2$ and a smooth point $p\in X$. Up to isomorphism there is only one local scheme of length two: $\Spec \CC[x]/(x^2)$, and  the embedded tangent space 
at $p$ may be identified with linear spans of such schemes, supported at $p$.
In particular, if $X$ is smooth then  $\fa_2(X)=\t(X)$, the tangential variety of $X$. 


Another, differential geometric,  definition of a tangent line  is as a limit of secant lines. 
This motivates the following.
\begin{definition}[Greater Areole]
The \textbf{$r$-th open greater areole} at $p$ is
   \begin{align*}
     \tilde{\mathfrak{a}}^\circ_r(X,p)&:= \bigcup_{
     \stack{x_j(t)\subset X}{x_j(t)\ra p} }\tlim_{t\ra 0}\langle x_1(t)\hd x_r(t)\rangle,
        \end{align*}
   the \textbf{$r$-th greater areole} at $p$ is the closure:
\begin{align*}
     \tilde{\mathfrak{a}}_r(X,p)&:=
     \ol{ \tilde{\mathfrak{a}}^\circ_r(X,p)} ,
   \end{align*}
and the \textbf{$r$-th greater areole variety} of $X$ is
$$
\tilde{\mathfrak{a}}_r(X):=\bigcup_{p\in X} \tilde{\mathfrak{a}}_r(X,p).
$$
\end{definition}
 \begin{remark}
 The difference between the areole and the greater areole is related to the difference of border rank and smoothable rank - 
 the latter was introduced in \cite{MR2842085}. Indeed, points in the $r$-th areole belong to a linear span of a scheme hence are of smoothable rank at most $r$. 
 \end{remark}

 The normal form in Lemma \ref{normalformlem} can be restated in the following way. 
If a point $v$ satisfies all assumptions of the Lemma, then it belongs to the $r$-th secant variety if and only if it belongs to the $r$-th greater areole $\tilde{\mathfrak{a}}_r(X,p)$ for a point $p\in\cO_{min}$.

\begin{lemma}\label{lem:contain}
${\mathfrak{a}}^\circ_r(X,p)\subset \tilde{\mathfrak{a}}^\circ_r(X,p)$
and
${\mathfrak{a}}_r(X,p)\subset \tilde{\mathfrak{a}}_r(X,p).$
\end{lemma}
\begin{proof}
It is enough to show the first inclusion. Let $v\in {\mathfrak{a}}^\circ_r(X,p)$. Then, by definition, there exists a scheme $S$, smoothable in $X$, supported at $p$ such that $v\in \langle S\rangle$. As $S$ is smoothable in $X$ we may find a family of points $x_i(t)$ for $i=1,\dots, k$, such that $S$ is their limit as $t\rightarrow 0$. We may also assume that $\langle x_1(t)\hd x_r(t)\rangle$ is of constant dimension. Let $E:=\tlim_{t\ra 0}\langle x_1(t)\hd x_r(t)\rangle$ that is also of dimension $r-1$. 
The linear span $\langle S\rangle$ of the limit is contained in the limit $E$ of linear spans.  By definition $E$ is contained in $\tilde{\mathfrak{a}}^\circ_r(X,p)$.
\end{proof}

\begin{remark}
The relation between the linear span of the limit and the limit
of linear spans can be viewed as a special case of upper semi-continuity of Betti numbers under deformation \cite[III.12.8]{MR0463157}, i.e.~the number of equations of fibers in a flat family of given degree can only jump up in the limit.
\end{remark}

The cases when the areole equals the greater areole are of particular interest. 

The following proposition is well-known, however usually stated in different language. It is essentially due to Grothendieck \cite{FGAIV} and played an important role in the construction of the Hilbert scheme. Recently, it was crucial in \cite{2015arXiv151105707B}. The proof of the following proposition follows from \cite[Theorem 5.7]{2015arXiv151105707B}.
\begin{proposition}\label{prop:eq}
Suppose that $p$ is a point of a variety $X$ embedded by at least an $(r-1)$-st Veronese embedding. Then:
$${\mathfrak{a}}_r(X,p)= \tilde{\mathfrak{a}}_r(X,p).$$
\end{proposition}

Motivated by applications, we restricted ourselves in the definition of areole to smoothable schemes. It may seem that the classification of such local schemes 
should be  easy, that  they should all be \lq almost\rq\  like $\Spec \BC[x]/(x^r)$. As we present below, the story is much more interesting. 

In the Hilbert scheme, 
the locus of schemes supported at $p$ and isomorphic to $\Spec \BC[x]/(x^r)$ is relatively open. Such schemes 
are   called \emph{aligned} \cite{MR1735271} or {\it curvilinear}. The schemes in the closure in the Hilbert scheme of the locus of aligned schemes, i.e., 
schemes that arise as degeneration of aligned schemes, are called {\it alignable}. For small values of $r$ all local smoothable schemes are alignable. 
An example of a scheme that is alignable, but not aligned is $\Spec \CC[x,y]/(x^2,y^2)$.   

A local scheme $\Spec \BC[x_1,\dots,x_n]/I$ is called \emph{Gorenstein} if the ideal $I$ is an \emph{apolar ideal} of a polynomial $f$ in the dual variables.
That is,  the  $x_i$ are differential operators on the dual space
of polynomials and  $I$ is the ideal of   differential operators annihilating $f$.

The Hilbert function of a local scheme with a maximal ideal $\fm$ assigns to $k$ the dimension of $\fm^k/\fm^{k+1}$. It is usually presented as a finite sequence,
as by convention, we omit the infinite string of  zeros after the last nonzero entry.

The last nonzero value of the Hilbert function for a Gorenstein scheme must be equal to $1$, because  if the form $f$ is of degree $d$,  the pairing with differential operators of degree $d$ provides a surjection onto $\BC$.

\begin{example}
The scheme $\Spec \BC[x,y]/(x^2,xy,y^2)$ is not Gorenstein, as its Hilbert function equals $(1,2)$. The scheme $\Spec \BC[x,y]/(x^2-y^2,xy)$ is Gorenstain as the ideal is apolar to $X^2+Y^2$, where $x(X)=1,x(Y)=0$ etc..
\end{example}

Schemes that are local and smoothable do not have to be alignable \cite{briancon}. Even more: there exist subvarieties of the Hilbert scheme corresponding to local, 
smoothable schemes that are of higher dimension than the component of alignable schemes. When $X$ is $n$ dimensional the dimension of locus of alignable 
schemes equals $(r-1)(n-1)$. Already for $r=12$ and $n=5$ there exists another family (also of dimension $44$) of smoothable schemes supported at $p$ that are not 
alignable. For $r=16$ and $n=7$ there exists a family of dimension $104$ of smoothable schemes  supported at $p$ \cite{BCR, JJ}. Summing over different $p$ we obtain 
a family of dimension $111=7\cdot 16 -1$, i.e.~a divisor in the Hilbert scheme! All this motivates one more definition.
\begin{definition}[Bud]
The \textbf{$r$-th open bud} at $p$ is
   \begin{align*}
     \mathfrak{b}^\circ_r(X,p)&:= { \bigcup \{ \langle R \rangle \mid   
                              R\simeq\Spec\CC[x]/(x^r)\text{\ and supported at } p \} },
   \end{align*}
    the \textbf{$r$-th bud} at $p$ is the closure:
   \begin{align*}
     \mathfrak{b}_r(X,p)&:= \overline{ \mathfrak{b}^\circ_r(X,p) },
   \end{align*}
and the \textbf{$r$-th   bud variety} of $X$ is
$$
 \mathfrak{b}_r(X):=\bigcup_{p\in X} \mathfrak{b}_r(X,p).
 $$
\end{definition}

Note  that the bud $\mathfrak{b}_r(X,p)$ contains the linear spans of all alignable schemes of given length that are supported on $p$. These schemes do not have to be Gorenstein. However,  of course  the aligned schemes are Gorenstein.

\begin{example}
In \cite{LMabten} we showed that when $\tdim A=\tdim B=\tdim C=m$,
then   $     \mathfrak{b}_m(Seg(\BP A\times \BP B\times \BP C))
=\ol{GL(A)\times GL(B)\times GL(C)\cdot [T_N]}\subset \BP (A\ot B\ot C)$,
where $T_N$ is a 
tensor such that  $T_N(A^*)\subset B\ot C$ corresponds to the centralizer of   a regular nilpotent element. 
\end{example}

Both areoles and   buds generalize the {\it tangent star} \cite{zak}, which is the case $r=2$.
\begin{proposition}
Let $p$ be a point of a   variety $X\subset \BP \VV$. Then
$$\mathfrak{b}_2(X,p)={\mathfrak{a}}_2(X,p)= \tilde{\mathfrak{a}}_2(X,p)=T^{\star}_pX,$$
where $T^{\star}_pX$ denotes the tangent star of $X$  at $p$.
\end{proposition}
\begin{proof}
The first equality follows by definition as any local scheme of length two is isomorphic to $\CC[x]/(x^2)$ i.e., aligned. 
The second equality is a special case of Proposition \ref{prop:eq} as any variety is its own first Veronese re-embedding. 
The third is just the definition of  the tangent star.
\end{proof}

Thus  when $r=2$, $\mathfrak{b}_2(X )={\mathfrak{a}}_2(X )= \tilde{\mathfrak{a}}_2(X )$. 
Moreover,     $\fb_2^0(X)=\fb_2(X)$ because all alignable schemes 
of length two are aligned. 

When $r=3$, \cite[Thm. 1.11]{MR3239293} shows that when $X=G/P$ is generalized cominuscule,
they still coincide:  $\mathfrak{b}_3(G/P )={\mathfrak{a}}_3(G/P )= \tilde{\mathfrak{a}}_3(G/P )$.  Moreover,  $\fb^0_3(G/P)=\fb_3(G/P)$ because all the points in the bud give 
aligned schemes.

\medskip

We now bound  the dimensions of all these  varieties. 
Recall that for an $n$-dimensional   variety,  $\tdim \s_r(X)\leq rn+r-1$.
\begin{prop}\label{prop:ineq}
Let  $p$ be a point of an $n$ dimensional   homogeneous variety
 $X\subset \BP^N$.
 Then
 \begin{align*}
  \dim \tilde{\mathfrak{a}}_r(X,p)&\leq rn-n+r-2,
\\
\dim \mathfrak{a}_r(X,p)&\leq rn-n+r-2,\\
\dim \mathfrak{b}_r(X,p)&\leq (r-1)n , 
 \end{align*}
 and hence, 
  \begin{align*}
   \dim \tilde{\mathfrak{a}}_r(X )&\leq rn +r-2,\\
\dim \mathfrak{a}_r(X )&\leq rn +r-2,\\
\dim \mathfrak{b}_r(X )&\leq  r n . \end{align*}
\end{prop}
 \begin{proof}
 The second inequality follows, simply by bounding the dimension of the locus of punctual smoothable schemes as a divisor in the Hilbert scheme. 
 The third equality follows as the locus of alignable schemes is of dimension $(r-1)(n-1)$. 
 
To prove the first inequality consider the projection $pr:S_r(X)\ra X^{(r)}\times G(r,N+1)$.
    The intersection of $pr(S_r(X))$ with the small diagonal in $X^{(r)}$ times the Grassmannian is at most a divisor. 
    Since $X$ is
    homogeneous, the fibers of the projection of the intersection to the small diagonal are all isomorphic, 
    hence are of dimension at most $nr-n-1$. The inequality follows.  
\end{proof}

\begin{remark} The only point in the proof where we used that $X$ was homogeneous
was to have equi-dimensional fibers. We expect that the inequalities remain true for any smooth $X$.
\end{remark}

While the areole and greater areole have the same expected dimension, 
in many cases (for example $r\leq 9$) the areole is of strictly smaller dimension than expected, often coinciding with the bud. Further, 
the areole and the greater areole are not expected to be irreducible. 
The problem of distinguishing between the areole and greater areole appears to be important and difficult. On the other hand the bud is irreducible, as the 
locus of aligned schemes is irreducible in the Hilbert scheme.

By the inequalities in Proposition \ref{prop:ineq},
$\tilde \fa_r(X), \fa_r(X)$, and $\fb_r(X)$ are all proper subvarieties of the secant variety
when $\s_r(X)$ has the expected dimension. Finding the equations of any of the above varieties when $r>2$, even in the case of Segre or Veronese varieties is another important and difficult challenge.

\subsection{The bud and local differential geometry}

We   thank Jaroslaw 
Buczy{\'n}ski and Joachim Jelisiejew for pointing us towards the following result:

\begin{proposition}Let $X\subset \BP \VV$ be a projective variety and  let  $p\in X$ be a smooth point.  Then
\begin{align*}
\fb_r(X,p)=
\overline{\bigcup_{
\stack{x(t)\subset X}{x (t)\ra p} }  \langle x(0),x'(0)\hd x^{(r-1)}(0)\rangle},  
\end{align*}
where the union is taken over all curves $x(\cdot)$ smooth at $p$.
\end{proposition}

In the language of differential geometry, $\fb_r(X,p)$ is the {\it  $(r-1)$-st osculating cone to $X$ at $p$}. Its span
is the $(r-1)$-st osculating space. 

\begin{remark} We obtain the same variety if we  take the union over analytic curves. 
\end{remark} 

\begin{proof}  
Given a smooth curve $C$, for each $r$, one has  the embedded  aligned scheme of length at most $r$ supported at $p$,  with
  span  $\langle x(0),x'(0)\hd x^{(r-1)}(0)\rangle$,  determined by it.

Given an aligned scheme $S$, we claim there exists a curve that contains it, is smooth at $p$,  and is contained in $X$. 
Let $\fm$ be the maximal ideal defining   $p$ in $\cO(X)_p$, the local ring of $p\in X$. 
Let $J$ be the ideal defining $S$ in $\cO(X)_p$. Since  the tangent space of $S$ is one-dimensional,  
  $(J+\fm^2)/\fm^2$ 
is a hyperplane in $\fm/\fm^2=T^*_pX$. Let $f_1,\dots,f_{\dim X-1}\in J$ span this hyperplane. 
Locally we may write $f_i=h_i/s_i$ with $h_i\in I(S)$ and  $s_i\in \BC[\VV]$ with $s_i(p)\neq 0$.
Let $U=X\backslash \cup_i \tzeros(s_i)$. Then the $h_i$ are lifts of the $f_i$ to $\BC[U]$. Consider the
subscheme (possibly reducible, non-reduced)   $Z\subset U$ they define. The Zariski tangent space $T_pZ$  is one
dimensional, so $Z$  must also be locally  one
dimensional at $p$ (at most one-dimensional because the local dimension is at most the dimension of $T_pZ$, 
  and at least because  we used $\tdim X-1$ equations), hence a
component through $p$ must be a curve, smooth at $p$, and this curve has
the desired properties.
\end{proof}

For a subvariety $X\subset \BP \VV$ and a smooth point
$x\in X$, there is a sequence of differential invariants
called the {\it fundamental forms} $\BF\BF_k: S^kT_xX\ra N^j_xX$,
where $N^j_xX$ is the $j$-th normal space. After
making choices of splittings and ignoring twists by line
bundles, write $\VV=\hat x\op T_xX\op N^2_xX\opc N^f_xX$.
See \cite[\S 2.2]{MR1966752} or \cite{MR3239293} for a quick introduction.
Adopt the notation $\BF\BF_1: T_xX\ra T_xX$ is the
identity map.

Let $X=G/P\subset \BP \VV$ be  generalized  cominuscule.
(This is a class of homogeneous varieties that includes  Grassmannians,
Veroneses and Segre varieties.) Then the only projective differential
invariants of $X$ at a point are the fundamental forms, and these
are easily (in fact pictorially) determined \cite{MR1966752}.

Let  $X$ be   generalized cominuscule, let $p=[v]$ and
let $v_1\hd v_{r-1}\in \hat T_pX$.
Then  calculations in \cite{MR3239293} show that  a general point of the bud $\fb_r(X,p)$ is
$$
[v+ \sum_{k=1}^{r-1}  \sum_{\ j_1+\cdots + j_k=r-1}
\BF\BF_k(v_{i_1}\hd v_{i_{r-1}})].
$$

\begin{example}
When $X=v_d(\BP W)$, and $p=[w^d]$, then elements of
$\hat T_pX$ are of the form $w^{d-1}u$ and
$$\BF\BF_k(w^{d-1}u_1\hd w^{d-1}u_k)=w^{d-k}u_1\cdots u_k.
$$
Thus a general point of the bud is of the form
$$
[\sum_{k=0}^{r-1}\sum_{\ i_1+\cdots + i_k=r-1}w^{d-k}u_{i_1}\cdots u_{i_{k}}].
$$
\end{example}

\begin{example}
The fundamental forms of Segre varieties are well-known. In particular,  for a $k$-factor Segre, the
last nonzero fundamental form is $\BF\BF_k$. The second fundamental form at $[a_1\otc a_k]$ is spanned by the quadrics
generating the ideal of $\BP (A_1/a_1)\sqcup \cdots \sqcup \BP (A_k/a_k)\subset \BP (A_1/a_1 \op \cdots \op 
  A_k/a_k)\simeq \BP T_{[a_1\otc a_k]}Seg(\BP A_1\ctimes \BP A_k)$. 
A general point of $\mathfrak{b}_r(Seg(\BP A_1\ctimes \BP A_k)),[a_1\otc a_k])$ is of the form
$$
[\sum_{i_1+\cdots + i_k \leq r-1} a_{1,i_1}\otc a_{k,i_k}]
$$
where $a_{j,i_j}\in A_j$  are arbitrary elements with $a_{j0}=a_j$.
\end{example}

\subsection{Examples of points  not in the open $r$-bud  of  generalized cominuscule varieties}
 
 Let $X\subset \BP \VV$ be generalized cominuscule.

Our examples are constructed from parametrized curves $x_j(t)$ in $X$. 
The general procedure to obtain the  
scheme to which the points degenerate as $t\rightarrow 0$ is as follows:
\begin{enumerate}
\item A Zariski open subset of $X$ has a rational parametrization, and 
{\it a priori} we are dealing with $r\tdim X$   different
$\BC[t]$ coefficients, but in practice
the number   is much smaller. For example, in the $k$-factor Segre, 
one is immediately reduced to $rk$ coefficients. Any scheme of length
$r$ can be embedded into a space of dimension $r-1$. So we work in
a space of dimension $r-1$. 

\item Find the ideal $I$ of polynomial equations that defines the curves as a parametric 
family over $\CC\times\CC^{r-1}$, where the first component corresponds to the variable $t$.
\item The desired scheme is given by  the ideal   $(I,t)$, which may be considered as a subscheme of $\BC^u$ for some $u\leq r-1$.    
\end{enumerate}
As our curves are given parametrically, the first two steps are instances of the implicitization problem. For the third,  one   substitutes $t=0$
into a set of  generators.  

\medskip

We already saw that the first possible example of a point not in the open bud is when $r=4$. 
 
Let  $r=4$,  and  consider
$p= \BF\BF_2(v_1,v_2)+v_3$, where $v_j\in T_xX$.
When $X=Seg(\BP A\times \BP B\times \BP C)$, 
$$
p= a_1\ot b_1\ot c_{4}+
a_1\ot b_{4}\ot c_1+a_{4}\ot b_1\ot c_1+\sum_{\s\in \FS_3}a_{\s(1)}\ot b_{\s(2)}\ot c_{\s(3)}.
$$
When $X=v_3(\BP W)$, then $p=xyz + x^2w$.

In the Segre case,  $p$ is in the span of the limit $4$-plane
of the  following four curves:   
\begin{align*}
x_0(t)&=a_1\ot b_1\ot c_1,\\
x_1(t)&=(a_1+ta_2+t^2a_4)\ot (b_1+tb_2+t^2b_4)\ot (c_1+tc_2+t^2c_4),\\
x_2(t)&=
(a_1+ta_3 )\ot (b_1+tb_3)\ot (c_1+tc_3)\\
x_3(t)&=
(a_1-t(a_2+a_3) )\ot (b_1-t(b_2+b_3))\ot (c_1-t(c_2+c_3)).
\end{align*}
If we set $b_j=c_j=a_j$, we obtain the corresponding curves in the Veronese $v_3(\BP A)$. 

Consider the affine open subset  $\CC^3\times \CC^3\times \CC^3$ where the coordinate $a_1\ot b_1\ot c_1$ is nonzero. 
All of the curves belong to it. Since the coefficients in $\BC[t]$ appearing
for each $j$ is the same for $a_j,b_j,c_j$, we may reduce to $\BC^3$.

We have reduced to four curves of the form:
$(y_1(t),y_2(t),y_3(t))$:  $(0,0,0)$, $(t,0,t^2)$, $(0,t,0)$, $(-t,-t,0)$. They satisfy the equation
$y_3=y_1(y_2+t)$, so  we may focus on the first two coordinates. Hence we have four points in the projective plane - a complete intersection of two quadrics.
The equations in $I$ are now of the form: 
$$y_1(y_1-t-2y_2), y_2(2y_1+t-y_2).$$ Substituting $t=0$ we obtain the annihilators of the nondegenerate quadratic form $Y_1^2+Y_1Y_2+Y_2^2$ (where $Y_j$ is
dual to $y_j$), i.e.,
the limiting scheme is isomorphic to $\Spec \CC[y_1,y_2]/(y_1y_2,y_1^2-y_2^2)$, that is Gorenstein alignable, but not aligned. 
 In particular, it is in the bud, but not the open bud. The Hilbert function equals $(1,2,1)$

\begin{example}[The Coppersmith-Winograd tensor]
The (second) Coppersmith-Winograd tensor generalizes the example above. It is
\be\label{tildeTcw}
\tilde T_{q,CW}:= \sum_{j=1}^q (a_0\ot b_j\ot c_j+ a_j\ot b_0\ot c_j+ a_j\ot
b_j\ot c_0 )
+a_0\ot b_0\ot c_{q+1}+ a_0\ot b_{q+1}\ot c_{0}+ a_{q+1}\ot b_0\ot c_0
\in \BC^{q+2}\ot \BC^{q+2}\ot \BC^{q+2}
\ene
It equals
\begin{align*}
\tlim_{t\ra 0} [  
&\sum_{i=1}^q \frac 1{t^2}(a_0+ta_i)\ot (b_0+tb_i)\ot (c_0+tc_i)\\
&-\frac 1{t^3}(a_0+t^2(\sum_{j=1}^qa_j))\ot (b_0+t^2(\sum_{j=1}^qb_j))\ot (c_0+t^2(\sum_{j=1}^qc_j))\\
&+[\frac 1{t^3}-\frac q{t^2}](a_0+t^3a_{q+1} )\ot (b_0+t^3b_{q+1})\ot (c_0+t^3a_{c+1})].
\end{align*}

The Coppersmith-Winograd tensors are symmetric, the first  corresponds to the polynomial
$x(y_1^2+\cdots + y_q^2)$, and the second (which is above), the polynomial
$x(xz+ y_1^2+\cdots + y_q^2)$. These polynomials have symmetric ranks respectively 
$2q+1$ and $2q+3$ (respectively shown in  \cite{MR2628829,2015arXiv151204905C}).
In \cite{2015arXiv150403732L} we showed these agree with their tensor ranks - thus the Comon conjecture \cite{Como02:oxford},
that the rank and symmetric rank of a symmetric tensor agree, holds for these tensors. 
Moreover, since our border rank algorithm is symmetric and matches the lower bound, the border rank version
of the Comon conjecture \cite{MR3092255} holds for these tensors as well.

Since the tensor is symmetric, we may immediately reduce to $\BC^{q+2}$ and work in the open set
where $a_0=1$. Then in the resulting $\BC^{q+1}$ the curves are:
$$
(t,0\hd 0), \ (0,t,0\hd 0)\hd (0\hd 0,t,0), (t^2\hd t^2,0), (0\hd 0,t^3)
$$
We see that $y_{q+1}=t^3-(\sum_{i=1}^qy_i)t^2+y_1y_2q-(\sum_{i=1}^qy_i)t/q+(\sum_{i=1}^qy_i^2)/q-y_1y_2/q$, so we are reduced to the curves
$$
(0\hd 0), \ (t,0\hd 0), \ (0,t,0\hd 0)\hd (0\hd 0,t), (t^2\hd t^2) 
$$
which satisfy the equations
$y_i(y_i-t)-y_j(y_j-t)$ and $y_i(ty_i-t^2-(t-1)y_j)$ for all $i\neq j$.
  Hence in the limit we obtain the Gorenstein scheme given by the annihilators of
  the nondegenerate quadric, namely  $\Spec \CC[y_1,\dots,y_{q-1}]/(y_iy_j,y_i^2-y_j^2)_{1\leq i<j\leq q-1}$.
\end{example} 

\begin{remark} The schemes that we obtain are artifacts of
the  border rank algorithms we choose and are not intrinsic to the  tensors. Even if we restrict to schemes/algorithms of minimal degree their uniqueness is related to generalized identifiability questions.
\end{remark}

\subsection{The bud and matrix multiplication}
 If one could make the stronger statement that
if $\ur(\Mn)=r$, then $\Mn\in \fb_r( Seg(\BP A\times \BP B\times \BP C),k))$ for some $k\in\cK$, then
one could probably prove  $\ur(\Mn)\geq 2\nnn^2-1$.
We calculated this for $\nnn=3$. The full proof involves many cases.
Here we present the two extreme cases  to
illustrate the idea.

Recall that a point of the open bud is of the form
$$
\sum_{3\leq i+j+k\leq r+2} a_i\ot b_j\ot c_k
$$
for some $a_i\in A$, $b_j\in B$, and $c_k\in C$.

\begin{proposition} Say $\Mn\in \fb^0_{2\nnn^2-2}( Seg(\BP A\times \BP B\times \BP C),k))$.
Then one cannot have any two of the three sets of vectors  $\{a_1\hd a_{\nnn^2}\}$, $\{b_1\hd b_{\nnn^2}\}$, $\{c_1\hd c_{\nnn^2}\}$,   linearly independent,
nor can one have any of  $\tdim \langle a_1\hd a_{2\nnn^2-\nnn-1}\rangle$,
$\tdim \langle b_1\hd b_{2\nnn^2-\nnn-1}\rangle$, or $\tdim \langle c_1\hd c_{2\nnn^2-\nnn-1}\rangle$ less than  $\nnn^2$.
\end{proposition}
\begin{proof}
For the first case, assume without loss of generality (by cyclic symmetry) that $a_1\hd a_{\nnn^2}$ and  $c_1\hd c_{\nnn^2}$ are  linearly independent.
Let $\a\in \langle a_1\hd a_{\nnn^2-1}\rangle\upperp$ be non-zero so
$\a(a_{\nnn^2})\neq 0$ and similarly $\gamma\in \langle c_1\hd c_{\nnn^2-1}\rangle\upperp$, $\gamma (c_{\nnn^2})\neq 0$.
Consider the proper   left ideal $\Mn\a\subset B$. (Here $A,B,C$ are all the same algebra of
$\nnn\times\nnn$ matrices, but we distinguish them for clarity.)
Then, viewing matrix multiplication as $C\times A\ra B$,  $\Mn\a\subset B$ is the span  
of $b_1\hd b_{\nnn^2-1}$ because  for each choice of $\tilde\g\in C$, the   vector:
$$
(\tilde \g(c_{\nnn^2-1})\hd \tilde \g(c_1))
\begin{pmatrix} \a(a_{\nnn^2}) & 0 & \cdots & 0\\
\a(a_{\nnn^2+1}) & \a(a_{\nnn^2}) & 0 & \cdots \\
\vdots &  & \ddots &   \\
\a(a_{2\nnn^2-2}) & \cdots & &  \a(a_{\nnn^2})
\end{pmatrix}\begin{pmatrix}b_1\\ b_2\\ \vdots \\ b_{\nnn^2-1}\end{pmatrix}\in B.
$$
appears in $\Mn\a$. 
  On the other hand,  the right ideal $\gamma\Mn\subset B$,  is  the span of the   vectors, for each choice of
  $\tilde \a\in A^*$, 
$$(b_{\nnn^2-1} \hd b_1 )\begin{pmatrix} \gamma(c_{\nnn^2}) & 0 & \cdots & 0\\
 \gamma(c_{\nnn^2+1}) &  \gamma(c_{\nnn^2}) & 0 & \cdots \\
\vdots &  & \ddots &   \\
\gamma(c_{2\nnn^2-2}) & \cdots & &  \gamma(c_{\nnn^2})
\end{pmatrix}
\begin{pmatrix}\tilde\a(a_1)\\ \tilde\a(a_2)\\ \vdots \\ \tilde\a(a_{\nnn^2-1})\end{pmatrix}, 
$$
which is also the span  
of $b_1\hd b_{\nnn^2-1}$, 
a contradiction.

For the second case,  assume $\tdim \langle a_1\hd a_{2\nnn^2-\nnn-1}\rangle<\nnn^2$. There exists a nonzero $\a\in \langle a_1\hd a_{2\nnn^2-\nnn-1}\rangle\upperp$.
Then the non-zero  left ideal $\a\Mn $ is a subspace of $\langle c_1\hd c_{\nnn-1}\rangle$, but the
smallest dimension of a left ideal is $\nnn$, a contradiction.
\end{proof}

\section{A border rank analog of the AFT theorem}\label{braftsect}
We recall the  Alexeev-Forbes-Tsimerman 
variant of the substitution method, as rephrased in \cite{LMabten}:
 
\begin{proposition}\cite[Appendix
B]{alexeev+forbes+tsimerman:2011:tensor-rank},  \cite[Chapter 6]{blaser2014explicit}\label{prop:slice}
Fix a basis $a_1\hd a_{\aaa}$ of $A$.
Write  $T=\sum_{i=1}^\aaa a_i\otimes M_i$, where $M_i\in B\otimes C$.
Let $\bold R(T)=r$ and $M_1\neq 0$. Then there exist constants $\lambda_2,\dots,
\lambda_m$, such that  the tensor
$$\tilde T:=\sum_{j=2}^m a_j\otimes(M_j-\lambda_j M_1)\in  a_1 \upperp \ot B\ot
C,$$
has rank at most $r-1$. Moreover, if $\trank (M_1)=1$ then for any
choices of  $\lambda_j$ we have $\bold R(\tilde T)\geq r-1$.
\end{proposition}

Here is  a border rank version:

\begin{lemma}\label{borderaft} Let $T\in A\ot B\ot C$. Then there exists
a hyperplane $H_A\subset A^*$ such that $\ur(T|_{H_A\times B^*\times C^*})\leq \ur(T) -1$.

In other words, there exists $a\in A$ such that image $T'$ of $T$ under the projection 
$A\ot B\ot C\ra (A/a)\ot B\ot C$ has $\ur(T')\leq \ur(T)-1$. 
\end{lemma}
\begin{proof}
First note this is true for rank (AFT). Let $T_t$ be a curve with $\tlim_{t\ra 0}T_t=T$.
For each $t$ the statement is true for some $H_t$, and by the proof in
\cite{LMabten}, one sees that  the $H_t$ will vary smoothly with $t$. Since projective space is compact there is
a limiting $H_0$   and the statement holds with $H_A=H_0$.
\end{proof}

One could in principle  apply this,  alternating the roles of $A,B,C$, to potentially obtain border rank bounds
up to $\tdim A+\tdim B+\tdim C-3$, but the need to test all possible hyperplanes makes such a use unlikely in practice. 

\begin{corollary}\label{mredchoice}  Let $\Mnred=\Mn- \sum_j x^1_\nnn\ot y^\nnn_j\ot z^j_1$ be a reduced matrix multiplication operator.
Then $\ur( \Mn)\geq \ur(\Mnred)+1$.
\end{corollary}

\begin{proof} There are $\nnn$ non-zero orbits in $A$ under $G_{\Mn}$. Say the element 
$x$ of $A$ that works in Lemma
\ref{borderaft} has rank $r$. If we 
act on it by an element of $G_{\Mn}$, the new element still works, so it will work for
points in the orbit closure   $\ol{G_{\Mn}\cdot x}$, which contain the rank one elements. 
  Finally all rank one elements are equivalent to $x^1_{\nnn}$.
\end{proof}

\section{A new lower bound for the border rank of matrix multiplication}

\begin{theorem}\label{mredbnd}
When $\nnn\geq 3$, $\ur(\Mnred)\geq 2\nnn^2-\nnn $.
\end{theorem}

When $\nnn=2$, it was shown in \cite{DBLP:journals/corr/LandsbergR15a} that $\ur(M_{\langle 2\rangle}^{red})=5$.

Theorem \ref{mredbnd} combined with Corollary \ref{mredchoice} implies:

\begin{theorem}\label{bndthm} Let $\nnn\geq 2$, then  $\ur(\Mn)\geq 2\nnn^2-\nnn+1$.
\end{theorem}

\begin{remark} The result was already shown for $\nnn=2$, where it is optimal,  in  \cite{MR2188132,MR3171099}.
\end{remark}

\begin{remark} The state of the art in other small cases is  
$16\leq \ur(\Mthree)\leq 20$ (the upper bound appears in \cite{MR3146566}), and $29\leq \ur(M_{\langle 4\rangle})\leq 49$ (with
the upper bound due to Strassen).
\end{remark}

\begin{proof}[Proof of Theorem \ref{mredbnd}]
We use the Koszul flattenings defined in \cite{MR3376667} which were used to prove $\ur(\Mn)\geq 2\nnn^2-\nnn$.
Write $ \Mnred\in   \BC^{\nnn^2-1}\ot B \ot C= \tilde A\ot   B \ot C$.
Consider for any $T\in \tilde A\ot B \ot C$, the Koszul flattening
$T^{\ww{\nnn-1}}_{\tilde A}: \La{\nnn-1}\tilde A\ot  B^*\ra \La\nnn \tilde A\ot C$  given by: $$T=\sum t^{ijk}a_i\ot b_j\ot c_k
f_1\ww\cdots \ww f_{n-1}\ot \b \mapsto \sum \b(b_j) t^{ijk}a_i\ww f_1\ww\cdots \ww f_{n-1}\ot c_k.$$
For matrix multiplication and for the reduced matrix multiplication this map
factors, $B=V^*\ot W$ and $C=W^*\ot U$, where $U,V,W=\BC^{\nnn}$ to 
$(M_{\langle 1,1,\nnn\rangle})^{\ww \nnn -1}_A\ot \Id_W$   and $(\Monred)^{\ww \nnn -1}_{\tilde A}\ot \Id_W$ respectively,
where $(\Monred)^{\ww \nnn -1 }_{\tilde A}: \La{\nnn-1}\tilde A\ot   V \ra \La\nnn \tilde A\ot U$. 
As in \cite{MR3376667}, we obtain the best result by restricting to an $A':=\BC^{2n-1}\subset \tilde A$.
Define the map 
\begin{align*}
\phi: \tilde A &\ra A'\\
x^i_j&\mapsto e_{i+j-1}.
\end{align*}

Since $(a\ot b\ot c)_{\BC^{2n-1}}^{\ww (\nnn-1)}$ has rank $\binom{2\nnn-2}{\nnn-1}$ (its image is $a\ww \La{\nnn-1}(\BC^{2n-1}/a)\ot c$),
it will suffice to prove, for $\nnn\geq 3$, that
$$
\frac{\trank (\Monred)^{\ww \nnn -1 }_{  A'}}{\binom{2\nnn-2}{\nnn-1}} \geq 2\nnn-1
$$
We claim that
$\trank (\Monred)^{\ww \nnn -1 }_{  A'}=\nnn\binom{2\nnn-1}\nnn -1$, which will prove the result when $\nnn\geq 3$.

To prove the claim we establish some notation.
Write $e_S=e_{s_1}\ww\cdots \ww e_{s_{\nnn-1}}$, where $S\subset [2\nnn-1]$ has cardinality $\nnn-1$.
Our map is
$$
e_S\ot v_k\mapsto \sum_{\{ m\mid ( m,k)\neq (\nnn,1)\} }  \phi(u^m\ot v_k)\ww e_S\ot u_m = \sum_{\{ m\mid ( m,k)\neq (\nnn,1)\} } e_{m+k-1}\ww e_S\ot u_m
$$
Index a basis of the source by pairs $(S,k)$, with $k\in [\nnn]$, and the target by $(P,l)$ where
$P\subset [2\nnn-1]$ has cardinality $\nnn $ and $l\in[\nnn]$.
We define an order relation on the target basis vectors in the following way. For $(P_1,l_1)$ and $(P_2,l_2)$, set $l=\tmin\{ l_1,l_2\}$, and declare
$(P_1,l_1)<(P_2,l_2)$ if and only if
\begin{enumerate}
\item In lexicographic order,  the set of $l$ minimal elements of $P_1$ is strictly after  the set of $l$ minimal elements of $P_2$ (i.e.~the smallest element of $P_2$ is smaller than the smallest of $P_1$ or they are equal and the second smallest of $P_2$ is smaller or equal etc.~up to $l$-th),  or
\item the $l$ minimal elements in $P_1$ and $P_2$ are the same, and $l_1<l_2$. 
\end{enumerate}
Note that
$(\{\nnn\hd 2\nnn-1\},1)$ is the unique minimal element for this relation and $([\nnn],\nnn)$ is the unique maximal element.
Note further that
$$
e_{n+1}\ww\cdots \ww e_{2\nnn-1}\ot u_n\mapsto
e_{n}\ww\cdots \ww e_{2\nnn-1}\ot v_1
$$
i.e., that 
$$
(\{\nnn+1\hd 2\nnn-1\},\nnn)\mapsto (\{\nnn\hd 2\nnn-1\},1)
$$
We will prove the claim by showing that the image is the span of all basis elements $(P,l)$ except the maximal element $([\nnn],\nnn)$.
We work by induction using the relation, the base case 
that  $(\{\nnn\hd 2\nnn-1\},1)$ is  in the image has been established.
Let $(P,l)$ be any basis element other than the maximal, and assume all
$(P',l')$ with  $(P',l')<(P,l)$ have been shown to be in the image.
Write $P=(p_1\hd p_{\nnn})$ with $p_i<p_{i+1}$. 
Consider the image of $(P\backslash \{ p_l\}, 1+p_l-l)$
which is 
$$\sum_{\{ m\mid ( m, { 1+p_l-l}) \neq ( \nnn, 1)\} }  \phi(u^m\ot v_{ 1+p_l-l})\ww e_{P\backslash \{ p_l\} } \ot u_m=
\sum_{\{ m\mid ( m, { 1+p_l-l}) \neq ( \nnn, 1)\} }   e_{ p_l-l+m}\ww e_{P\backslash \{ p_l\}} \ot u_m.
$$
In particular, taking $m=l$ we see $(P,l)$ is among the summands, as long as  $(P,l)$ is not the maximal element.
If $m<l$, the contribution to the summand is a $(P',m)$ where the first $m$ terms of $P'$ equal the
first of $P$, so by     condition (2),  $(P',m)<(P,l)$. If $m>l$, the summand is a $(P'',m)$ where
the first $l-1$ terms of $P$ and $P''$ agree, and the $l$-th terms are respectively
$p_l$ and $p_l-l+m$ so by condition (1) $(P'',m)<(P,l)$.
\end{proof}

\bibliographystyle{amsplain}
 
\bibliography{Lmatrix}

\end{document}